\documentclass[11pt]{amsart}
\usepackage{amssymb,amsmath,amsthm,amsfonts,amsopn,calligra,color,enumerate,latexsym,mathtools,microtype,stackrel}
\usepackage[all,cmtip]{xy} \SelectTips{eu}{} \SilentMatrices
\xyoption{all}
\xyoption{arc}
\usepackage[T1]{fontenc}
\usepackage[normalem]{ulem}
\usepackage{enumitem}
\usepackage{amscd}
\numberwithin{equation}{section}

\usepackage[urlcolor=blue]{hyperref}

\makeatletter  % for 2020 subject classification
\@namedef{subjclassname@2020}{%
  \textup{2020} Mathematics Subject Classification}
\makeatother

\usepackage{color}
\definecolor{tealgreen}{rgb}{0.0, 0.6, 0.5}

\theoremstyle{plain}
\newtheorem{thm}{Theorem}[section]

\newtheorem{prop}[thm]{Proposition}
\newtheorem{lemma}[thm]{Lemma}
\newtheorem{cor}[thm]{Corollary}

\theoremstyle{definition}
\newtheorem{defn}[thm]{Definition}
\newtheorem*{defn*}{Definition}
\newtheorem*{question*}{Question}

\newtheorem{example}[thm]{Example}
\newtheorem*{example*}{Example}
\newtheorem{rem}[thm]{Remark}
\newtheorem*{rem*}{Remark}

\newtheorem{ipg}[thm]{}

\newcommand{\leftrarrows}{\mathrel{\raise.75ex\hbox{\oalign{%
   $\scriptstyle\leftarrow$\cr
   \vrule width0pt height.5ex$\hfil\scriptstyle\relbar$\cr}}}}
 \newcommand{\lrightarrows}{\mathrel{\raise.75ex\hbox{\oalign{%
  $\scriptstyle\relbar$\hfil\cr
   $\scriptstyle\vrule width0pt height.5ex\smash\rightarrow$\cr}}}}
 \newcommand{\Rrelbar}{\mathrel{\raise.75ex\hbox{\oalign{%
   $\scriptstyle\relbar$\cr
  \vrule width0pt height.5ex$\scriptstyle\relbar$}}}}

\newcommand{\field}[1]{\mathbb{#1}}

\newcommand{\F}{\field{F}}

\renewcommand{\L}{\field{L}}
\newcommand{\PP}{\field{P}}
\newcommand{\X}{\field{X}}
\newcommand{\ideal}[1]{\mathfrak{#1}}
\newcommand{\m}{\ideal{m}}

\newcommand{\ffunc}[1]{\mathrm{#1}}
\newcommand{\func}[1]{\mathrm{#1} \,}
\newcommand{\fm}{\mathfrak{m}}

\DeclareMathOperator{\op}{op}
\newcommand{\lra}{\longrightarrow}
\newcommand{\xra}{\xrightarrow}
\newcommand{\dperp}{!}
\newcommand{\rank}{\func{rank}}

\newcommand{\Tor}{\ffunc{Tor}}

\newcommand{\im}{\func{im}}

\newcommand{\ra}{\rightarrow}

\DeclareMathOperator{\Hom}{Hom}

\newcommand{\be}{\begin{enumerate}}
\newcommand{\ee}{\end{enumerate}}

\newcommand{\li}%{{\mathrm{lic}}}
 {\leftfootline}

\renewcommand{\phi}{\varphi}

\DeclareMathOperator{\HH}{H}

\let\int\relax
\DeclareMathOperator{\int}{i}

\DeclareMathOperator{\id}{id}
\DeclareMathOperator{\Span}{Span}

\author[Faber]{Eleonore Faber}
\address{School of Mathematics \\ University of Leeds \\ Leeds, LS2 9JT, UK }
\email{e.m.faber@leeds.ac.uk}

\author[Juhnke-Kubitzke]{Martina Juhnke-Kubitzke}
\address{Department of Mathematics \\University of Osnabr\"uck  \\ Germany }
\email{juhnke-kubitzke@uni-osnabrueck.de}

\author[Lindo]{Haydee Lindo}
\address{Department of Mathematics \\ Harvey Mudd College \\ Claremont, CA 91711 }
\email{hlindo@hmc.edu}

\author[Miller]{Claudia Miller}
\address{Mathematics Department \\ Syracuse University \\ Syracuse, NY 13244}
\email{clamille@syr.edu}

\author[R.G.]{Rebecca R.G.}
\address{Department of Mathematical Sciences \\ George Mason University \\ Fairfax, VA  22030}
\email{rrebhuhn@gmu.edu}

\author[Seceleanu]{Alexandra Seceleanu}
\address{Department of Mathematics \\ University of Nebraska--Lincoln \\ Lincoln, NE  68588}
\email{aseceleanu@unl.edu}

\title{Canonical Resolutions over Koszul Algebras}
\subjclass[2020]{Primary: 16E05, Secondary: 13D02.}
\keywords{Koszul algebra, minimal free resolution, Betti numbers}
\date{}

\begin{document}
\begin{abstract}
We generalize Buchsbaum and Eisenbud's resolutions for the powers of the maximal ideal of a polynomial ring to resolve powers of the homogeneous maximal ideal over graded Koszul algebras. Our approach has the advantage of producing resolutions that are both more explicit and minimal compared to those previously discovered  by Green and Mart\'{\i}nez-Villa \cite{GreenMartinezVilla} or Mart\'{\i}nez-Villa and Zacharia \cite{MartinezVillaZacharia}.
\end{abstract}

\maketitle
\setcounter{tocdepth}{1} 
\tableofcontents

%%%%%%%%%%%%%%%%%%%%%%%%%
\section{Introduction}
\label{sec-intro}

Koszul algebras show up naturally and abundantly in algebra and topology. They were first introduced by Priddy in 1970  as algebras for which the bar resolution, which is normally far from minimal, admits a reduction to a comparatively small subcomplex; see \cite{Priddy}. Priddy's work explained contemporaneous ideas on restricted Lie algebras in the work of May and, separately,  that of Bousfield, Curtis, Kan, Quillen, Rector, and Schlesinger; see \cite{May, BCKQRS}. Priddy was an algebraic topologist, but Koszul algebras have since  been linked to several fundamental concepts across mathematics where they appear naturally and are studied extensively in fields as diverse as topology \cite{GKM}, representation theory \cite{BGS}, commutative algebra \cite{Froberg}, algebraic geometry  \cite{BGG}, noncommutative geometry \cite{Manin}, and number theory \cite{Positselski14}. 
For a general overview, see the monograph by Polishchuk and Positselski \cite{Polishchuk}.

An interesting feature of Koszul algebras is that they appear in pairs: 
every Koszul algebra $A$ has a dual algebra $A^!$ which is also a Koszul algebra (see Section \ref{sec-background}). 
 The prototypical example of such a Koszul pair is a polynomial algebra $S$ over a field, together with the corresponding exterior algebra $\Lambda$. The associated theory of Koszul duality is a generalization of the duality underlying the Bernstein--Gelfand--Gelfand correspondence \cite{BGG} describing coherent sheaves on projective space in terms of modules over the exterior algebra. This exemplifies the philosophy that facts relating the symmetric and exterior algebras often have Koszul duality counterparts.

In this paper we extend Priddy's methods of constructing free resolutions over standard graded Koszul algebras and generalize Buchsbaum and Eisenbud's resolutions in \cite{BE} to resolve powers of the  homogeneous maximal ideal over Koszul algebras. Resolutions over Koszul algebras have previously appeared in works of Green and Mart\'{\i}nez-Villa \cite[Theorem 5.6]{GreenMartinezVilla}, Mart\'{\i}nez-Villa and Zacharia \cite[Proposition 3.2]{MartinezVillaZacharia} and in other sources referenced below. Our approach has the advantage of producing explicit minimal resolutions.

In particular, in  \cite{Priddy} Priddy exploits a natural differential on $A\otimes_k A^!$ to give an explicit construction for the linear minimal graded free resolution of the residue field of a graded Koszul algebra; see Definition \ref{def:Priddycx}. In this paper, we extend this construction to a family of acyclic complexes that yields highly structured resolutions of the powers of the homogeneous maximal ideal over  standard graded  Koszul algebras; see Definition \ref{def:Fcx}. Since these complexes are typically not minimal, we also seek to determine their minimal counterparts. To achieve this, we take inspiration from results that describe structured resolutions over a polynomial ring $S$ constructed starting from the Koszul complex (exterior algebra) $\Lambda$ and its generalizations; see \cite{BE}. We provide analogs of these results using any pair of Koszul dual algebras, $A$ and $A^!$, instead of $S$ and $\Lambda$.

Our main result generalizes the canonical resolutions for the powers of the homogeneous maximal ideal constructed over a polynomial ring $S$ by Buchsbaum and Eisenbud in \cite{BE} to obtain minimal free resolutions for powers of the homogeneous maximal ideal of a graded  Koszul algebra. In contrast to the situation over $S$, these are in general infinite resolutions. This allows us to obtain an explicit formula for the graded Betti numbers defined by  $\beta_{i,j}(\m^a)=\dim_k\Tor_i(\m^a,k)_j$ and the graded  Poincar\'e series $P^A_{\m^a}(y,z)=\sum_{i,j\geq 0}\beta_{i,j}(\m^a) y^jz^i$.
The following is a combination of Theorem \ref{thm:minimalres} and Corollary~\ref{cor:betti}. 

\vspace{2mm}

\noindent 
{\bf Theorem.}
{\it If $A$ is a graded Koszul algebra with homogeneous maximal ideal $\fm$, the complexes  
\begin{equation*}
\mathbb L^A_a:\ \ 
\cdots \to L^A_{n,a}
\xra{\partial_{n}'}
L^A_{n-1,a}
\xra{\partial_{n-1}'} 
\dots 
\xra{\partial_{1}'}
L^A_{0,a} \xra{\varepsilon_a}\m^a\to 0 , 
\end{equation*}
defined in equation \eqref{Lcomplex} with the augmentation map $\varepsilon_a$ defined in equation \eqref{epsilon} are minimal free resolutions of the powers $\m^a$ with $a\geq 1$.

The nonzero graded Betti numbers of the powers of $\m$ are given by
\[
\beta_{n,n+a}^A(\m^a)=\sum_{i=1}^a (-1)^{i+1} \dim_k( (A^!)^*_{n+i})\dim_k(A_{a-i}).
\]
and the graded Poincar\'e series is
\[
P^A_{\m^a}(y,z)=-(-z)^{-a}H_{{A^!}^*}(yz) H_{A/\m^a}(-yz).
\]
}
 In particular, the minimal graded resolution of $\m^a$ is $a$-linear. 

\vspace{2mm}

The Betti numbers presented in the theorem are recovered in a more restricted setting in the recent paper \cite{vandebogert} investigating resolutions of monomial ideals over strongly Koszul algebras via different techniques. 

The paper is structured as follows. In Section~\ref{sec-background}, we provide background on Koszul algebras and the Priddy complex. In Section~\ref{sec-enveloping}, we explain how to obtain a free solution of a module $M$ from a resolution of the ring over its enveloping algebra. In Section \ref{sec-koszul}, we rewrite this resolution as the totalization of a double complex, in the case that the module is a power of the homogeneous maximal ideal and the ring is a Koszul algebra. In Section \ref{sec-minimal}, we give the minimal resolution and Betti numbers for the powers of the homogeneous maximal ideal over a Koszul algebra.
In Section~\ref{sec-examples} we apply our construction to several specific Koszul algebras $A$ to obtain explicit formulas for the Betti numbers of $\m^a$.

\section{Koszul algebras and the Priddy resolution}
\label{sec-background}

Throughout $k$ is a field and $A$ is a graded $k$-algebra having finite-dimensional graded components with $A_i=0$ for $i<0$ and $A_0=k$. 
We further assume that $A$ is standard graded, that is, $A$ is generated by $A_1$ as an algebra over $A_0=k$.

\begin{defn}[{\cite[Chapter 2]{Priddy}}]
%Let $k$ be a field and $A$ a graded $k$-algebra. 
We say that $A$ is \textit{Koszul} if $k=A/A_{>0}$ admits a linear graded free resolution over $A$, i.e., a graded free resolution $P_\bullet$ in which $P_i$ is generated in degree $i$.
%the following equivalent conditions hold:
% \begin{enumerate}
%     \item $\Ext_A^{ij}(k,k)=0$ for $i \ne j$;
%     \item $A$ is one-generated and the algebra $\Ext_A^*(k,k)$ is generated by $\Ext_A^1(k,k)$;
%     \item $A$ is quadratic and $\Ext_A^*(k,k) \simeq A^!$;
%     \item the algebra $\Ext_A(k,k)$ equipped with internal grading is one-generated.
% \end{enumerate}
\end{defn}

Classes of graded Koszul algebras arise from: quadratic complete intersections \cite{Tate}, quotients of a polynomial ring by quadratic monomial ideals \cite{Froberg}, quotients of a polynomial ring by homogeneous ideals which have a  quadratic Gr\"obner basis, and from Koszul filtrations \cite{Koszulfiltration}; 
see, for example, the survey paper by Conca \cite{ConcaSurvey}.

We now focus on quadratic algebras in order to define Koszul duality.

\begin{defn}[{\cite[Chapter 1, Section 2]{Polishchuk}}] %page 6 of Polischuk
Let $A$ be a standard graded $k$-algebra. We say that $A$ is quadratic if $A=T(V)/Q$, where $V$ is a $k$-vector space, $T(V)$ is the tensor algebra of $V$, and $Q$ is a quadratic ideal of $T(V)$.

%We say that $A$ is \textit{one-generated} if the multiplication map $p$ from the tensor algebra $T(A_1)$ generated by $A_1$ to $A$ is surjective.

%A one-generated algebra is called \textit{quadratic} if $\ker(p)$ is generated as a two-sided ideal in $T(A_1)$ by its subspace $I_A=J_A \cap T^2(A_1) \subseteq A_1 \otimes A_1$.

If $A$ is a quadratic algebra, its \textit{quadratic dual algebra} is defined by
\[
A^!=\frac{T(V^*)}{Q^\perp}
\] 
where $V^*=\Hom_k(V,k)$ and $Q^\perp$ is the quadratic ideal generated by the orthogonal complement to $Q_2$ in $T(V^*)_2 = V^* \otimes_k V^*$ with respect to the natural pairing between $V \otimes V$ and $V^* \otimes V^*$ given by
\[\langle v_1 \otimes v_2,v_1^* \otimes v_2^* \rangle=\langle v_1,v_1^*\rangle \langle v_2,v_2^* \rangle.
\]

Choosing dual bases $x_1,\ldots, x_d$ and $x_1^*,\ldots, x_d^*$ for $V$ and $V^*$ respectively yields that $T(V)=k\langle x_1,\ldots, x_d \rangle$ and $T(V^*)=k\langle x_1^*,\ldots, x_d^*\rangle $ are polynomial rings in noncommuting variables of  degrees $|x_i|=1$ and $|x_i^*|=-1$. This allows one to compute $Q^\perp$ given a quadratic ideal $Q\subseteq T(V)$ using linear algebra, as described for example in \cite[Section 8]{MP}.
\end{defn}

Graded Koszul algebras are quadratic (see, for example, \cite[Chapter 2, Definition 1]{Polishchuk}) and the duality of quadratic algebras restricts well to the class of Koszul algebras since $A$ and $A^!$ are Koszul simultaneously \cite[Chapter 2, Corollary 3.2 ]{Polishchuk}. Moreover,  $(A^!)^!=A$.

\begin{example}
The main example of Koszul dual algebras is given by the symmetric algebra on a vector space $V$
\[
S=k[x_1,\ldots, x_d]=\frac{k\langle x_1,\ldots, x_d \rangle}{( x_ix_j-x_jx_i, 1\leq i<j\leq d)}
\]
and the exterior algebra on $V^*$
\[
S^!=\Lambda=\frac{k\langle x_1^*,\ldots, x_d^* \rangle}{( (x_i^*)^2, x_i^*x_j^*+x_j^*x_i^*, 1\leq i\leq j\leq d)}.
\]
\end{example}

\begin{example}
\label{ex:0}
For the following commutative Koszul algebra  
\[A =\frac{ k[x,y,z]}{(x^2,xy,y^2)}=\frac{k\langle x,y,z \rangle}{(x^2,xy,y^2, xz-zx, xy-yx, yz-zy)},\]
the dual algebra is given by 
\[A^!=\frac{k\langle x^*, y^*, z^*\rangle}{((z^*)^2, x^*z^*+z^*x^*, y^*z^*+z^*y^*)}.\]
This pair of algebras are further discussed in Example \ref{ex:1}.
\end{example}

\begin{defn}
\label{def:Priddycx}
The {\em Priddy complex} \cite{Priddy} of a quadratic algebra $A$ is the complex $P^A_\bullet$ whose $i$-th term is given by
\[
P^A_i= A \otimes_k  {(A^!)}_i^{*} ,
\]
and the differential is defined by right multiplication by the trace element $\sum_{i=0}^d x_i\otimes x_i^*$, 
where multiplication by $x_i^*\in A^!$ on $(A^!)^*$ is defined as the dual of multiplication by $x_i^*$ on $A^!$.
\end{defn}

\begin{ipg}
\label{Priddyres}
The importance of the Priddy complex lies in the fact that $P^A_\bullet$ is acyclic if and only if $A$ is Koszul; see \cite[Chapter 2, Corollary 3.2]{Polishchuk}. Moreover, when $A$ is Koszul the Priddy complex, also called the generalized Koszul resolution, is a minimal free resolution of the residue field $k$ of $A$. This will be the base case in the proof that our construction in Section~\ref{sec-koszul} is a resolution.
\end{ipg}

\begin{ipg}
\label{Koszulduality}
Duality of Koszul algebras extends to an equivalence of derived categories that goes back to \cite{BGSch} and was developed further in \cite{BGS}. Let $T = A \otimes_k A^!$ which is an $A$-$A^!$-bimodule. For complexes $N_\bullet$ of $A^!$-modules and $M_\bullet$ of $A$-modules  define functors $L(N_\bullet)=T\otimes_{A^!} N_\bullet\cong A\otimes_k N_\bullet$ and $R(M_\bullet)=\Hom_A(T,M_\bullet)\cong \Hom_k(A^!,M_\bullet)\cong (A^!)^*\otimes_k M_\bullet$. It is shown in \cite[Theorem 2.12.1]{BGS} that these functors induce an equivalence of categories
\[
L : D^{\uparrow}(A^!) \rightleftarrows D^{\downarrow}(A) : R
\]
where $D^\uparrow(A^!)$ stands for the derived category of complexes $N_\bullet$ of graded $A^!$-modules  with $N_{i,j}=0$ for $i\gg 0$ or $i+j\ll0$  and $D^\downarrow(A)$  is the derived category of complexes $M_\bullet$ of graded $A$-modules with $M_{i,j}=0$ for $i\ll0$ or $i+j\gg 0$.

\end{ipg}

%%%%%%%%%%%%%%%%%%%%%%%%%

%%%%%%%%%%%%%%%%%%%%%%%%%
\section{Resolutions via the enveloping algebra}
\label{sec-enveloping}
%%%%%%%%%%%%%%%%%%%%%%%%%

Let $A$ be a (not necessarily commutative) $k$-algebra where $k$ is a field. 
In this section, we review how one obtains a free resolution of any $A$-module $M$ from a resolution of $A$ over its enveloping algebra. 
In general, one obtains a resolution that is far from minimal. We remedy this in Sections~\ref{sec-koszul} and ~\ref{sec-minimal} over Koszul algebras $A$ for the modules $\fm^a$ (and hence $A/\fm^a$). 

\begin{ipg}
Given a $k$-algebra $A$, its enveloping algebra is given by $A^e=A \otimes_k A^{\op}$. A left $A^e$-module structure is equivalent to an $A$-$A$-bimodule structure via
\begin{equation}
\label{eq:Aeaction}
(a \otimes b)\cdot m = a \cdot m \cdot b
\end{equation}
We consider $A$ as an $A^e$-module via the action in \eqref{eq:Aeaction}, where $a,b, m\in A$ and $ a \cdot m \cdot b$ represents internal multiplication in $A$.

Consider a graded free resolution\footnote{One can always take the bar resolution of $A$ over its enveloping algebra, but that is usually far from minimal.}
of $A$ over $A^e$ and note that  
any free left $A^e$-module $F$ can be rewritten as 
\[
F = A^e \otimes_k V 
=
A \otimes_k A^{\op} \otimes_k V
\cong
A \otimes_k V \otimes_k A
\] 
for some vector space $V$, where the rightmost expression is thought of as an $A$-$A$-bimodule via the outside two factors. 
Thus the resolution will be of the form 
\begin{equation}
\label{eq:resAoverAe}
    \cdots \to A \otimes_k V_2 \otimes_k A\to A \otimes_k V_1 \otimes_k A \to A \otimes_k A \xra{\varepsilon} A \to 0,
\end{equation}
where the augmentation $\varepsilon$ from $A \otimes_k A$ to $A$ is given by multiplication across the tensor.

We observe that $A \otimes_k k \otimes_k A \cong A \otimes_k A$. Thus setting $V_0=k$, we may write the resolution as a quasi-isomorphism of $A^e$-modules
\[
A \otimes_k V_\bullet \otimes_k A \xra{\simeq} A
\]
\end{ipg}

Next we show how to construct an $A$-free resolution for arbitrary $A$-modules $M$ using \eqref{eq:resAoverAe}.
This is well known; we include it because the construction is the basis of our next step in Section~\ref{sec-koszul}. 
In the case of Koszul algebras, it can also be seen using Koszul duality, and, more generally, it follows when the resolution of $A$ comes from an acyclic twisting cochain; see the remarks following the proof. 

\begin{prop}
\label{prop:resMoverA}
If $M$ is a graded $A$-module and $A \otimes_k V_\bullet \otimes_k A \xra{\simeq} A$ is a graded $A^e$-free resolution of $A$, then the induced map $A \otimes_k V_\bullet \otimes_k M \xra{\simeq} M$ is a graded $A$-free resolution of $M$ where the $A$-module structure on the latter tensor product is via the first factor.
\end{prop}

\begin{proof}
First note that both  $A \otimes_k V_\bullet \otimes_k A$ and $A$ are $A$-$A$-bimodules in the obvious ways. 
Furthermore, the complex $A \otimes_k V_\bullet \otimes_k A$ (considered as an $A$-module via its righthand factor) and the trivial complex $A$ both consist of free $A$-modules (although the latter is not a free $A^e$-module). Therefore the quasi-isomorphism $A \otimes_k V_\bullet \otimes_k A \xra{\simeq} A$ is actually a homotopy equivalence of $A$-modules. Hence it remains a quasi-isomorphism after tensoring over $A$ on the right with arbitrary $A$-modules. 
To see this, note that the augmented complex of free $A$-modules \eqref{eq:resAoverAe} is contractible, that is, homotopy equivalent to 0 (equivalently, it is split exact over $A$). But this complex is the mapping cone of the chain map $A \otimes_k V_\bullet \otimes_k A \ra A$. 

Therefore, upon tensoring \eqref{eq:resAoverAe} on the right over $A$ with a left $A$-module $M$, one obtains a quasi-isomorphism of left $A$-modules 
\begin{align*}
(A \otimes_k V_\bullet \otimes_k A) \otimes_A M &\xra{\simeq} A \otimes_A M 
\\
A \otimes_k V_\bullet \otimes_k M 
&\xra{\simeq} M 
\end{align*}
As the original resolution of $A$ was a map of $A$-$A$-bimodules, this one is still a map of $A$-modules (via the lefthand factor of the tensor product). 
Viewing $V_\bullet \otimes_k M$ as a (rather large) $k$-vector space, one sees that the complex on the left consists of free $A$-modules, giving a free $A$-resolution of $M$.
\end{proof}

\begin{rem}
\label{viakoszulduality}
Under the assumption that $A$ is a Koszul algebra, we include here an alternate proof of Proposition \ref {prop:resMoverA} using Koszul duality.
Using the functors from the equivalence described in \ref{Koszulduality}, for a graded $A$-module $M$,
one gets that $L(R(M)) \simeq M$. On the other hand, one computes that $R(M)$ is the complex

\[
0 \to (A^!)^* \otimes_k M_0 
\to (A^!)^* \otimes_k M_1 
\to \cdots 
\to (A^!)^* \otimes_k M_{i} 
\to \cdots
\]
Furthermore, it is clear from the definition that $L((A^!)^*)$ is simply the Priddy complex $P^A_\bullet$. Applying this to the complex above termwise and totalizing gives that  $L(R(M))$ equals the totalization of 
\[
0 \to P^A_\bullet \otimes_k M_0 
\to P^A_\bullet \otimes_k M_1 
\to \cdots 
\to P^A_\bullet \otimes_k M_{i} 
\to \cdots
\]
which is exactly the complex described in Proposition~\ref{prop:resMoverA}. 
In the case that $M=A/\fm^a$, one gets the double complex $\mathbb X^A_a$ described in Corollary~\ref{truncateddiagramhere}.  
\end{rem}

\begin{rem}
\label{viatwisting}
More generally, we now briefly describe this from the perspective of acyclic twisting cochains. Although these go far back, for recent quite general versions of the duality they afford, modeled on that of Dwyer, Greenlees, and Iyengar in \cite{DGI-06} and generalizing Koszul duality, and for descriptions of how it specializes to the situation of Koszul algebras, as well as the terms used below, see Avramov's paper \cite{Avramov}, especially Theorem 4.7. 
 
Let $A$ be an augmented dg (differential graded) algebra. When there is an augmented dg coalgebra $C$ with a map $\tau\colon C \to A$ of degree $-1$ that is a {\it twisting cochain}, that is, a Maurer-Cartan equation 
\[
\partial_A \tau + \tau \partial_C + \mu (\tau \otimes \tau) \Delta =0,
\]
holds, where $\Delta\colon C \to C\otimes C$ is the diagonal map and $\mu$ is the multiplication map, then one can form tensor products whose differential is ``twisted'' by $\tau$ yielding a natural map $A \otimes {}_\tau C_\tau \otimes A \to A$. If this is a quasi-isomorphism, then $\tau$ is called {\it acyclic}, in which case the induced map $A \otimes {}_\tau C_\tau \otimes M \to M$ is a quasi-isomorphism for all dg $A$-modules $M$ and a duality generalizing Koszul duality holds. 
An example is given by the bar construction $C=BA$ with the canonical map $\tau\colon BA \to A$, but in the case of Koszul algebras one can get by with a much smaller complex using Priddy's construction. 
\end{rem} 

\begin{rem}
Suppose $A$ is local (or standard graded) $k$-algebra with (homogeneous) maximal ideal $\fm$. The resolutions obtained in Proposition~\ref{prop:resMoverA} are in general not minimal (respectively, minimal graded) resolutions even when one starts with a minimal (respectively, minimal graded) resolution of $A$ over $A^e$. 
For example, for the explicit resolution $\mathbb X^A_a$ given in Corollary~\ref{truncateddiagramhere}, although $\partial'$ is minimal, $\partial''$ is clearly not. 
\end{rem}

%%%%%%%%%%%%%%%%%%%%%%%%%
\section{The case of Koszul algebras}
\label{sec-koszul}
%%%%%%%%%%%%%%%%%%%%%%%%%

In this section, under the further assumption that $A$ is a Koszul $k$-algebra, we write the $A$-free resolution of $A/\fm^a$ obtained in Remark \ref{viakoszulduality} as the  totalization of a certain double complex. 

First we recall the minimal graded resolution of $A$ over $A^e$ following the presentation in \cite[Section 3]{vandenBergh}. 
It is a symmetrization of the resolution   of $k$ over $A$ found by Priddy in \cite{Priddy}, which is presented in Definition \ref{def:Priddycx}. 

\begin{defn}
\label{def:Fcx}
Let $A$ be a Koszul $k$-algebra with dual Koszul algebra $A^!$. 
Let $(A^!)^*=\Hom_k(A^!,k)$; thus $(A^!)^*$ is an $A^!$-module where the action of $A^!$ on $(A^!)^*$ is the dual of the action of $A^!$ on itself. 
Define free $A$-modules
\[
F_n 
= A \otimes_k (A^!)^*_n \otimes_k A 
% = \underbrace{A \otimes_k (A^!)^*_n}_{\partial_1} \otimes_k A 
% =
% A \otimes_k \underbrace{(A^!)^*_n \otimes_k A}_{\partial_2}
\]
and differentials
\[
\partial_n = (\partial')_n + (-1)^n (\partial'')_n 
\] 
where  
\[
\partial' 
={\textrm{ right multiplication by }}
\sum_{i=0}^d x_i \otimes x_i^* \otimes 1
\]
(which will form our vertical maps)
and 
\[
\partial''
={\textrm{ left multiplication by }}
\sum_{i=0}^d 1 \otimes x_i^* \otimes x_i
\]
(which will form our horizontal maps)\footnote{There is a misprint in \cite{vandenBergh} with regards to this map.}.
%as 
%\[
%{\cm S \otimes_k \Lambda }
%\cong S \otimes_k \Lambda(V)
%\cong \Lambda_S(S\otimes_k V)
%{\textrm{ and }} 
%S\otimes_k V \cong S^n.
%\]   
Then the complex 
\begin{equation}
\label{eq:Fcomplex}
    \F^A: \cdots \to F_n\xra{\partial_n} F_{n-1}\to \cdots F_0 \xra{\varepsilon} A \to 0
\end{equation} 
  augmented by the multiplication map $\varepsilon$ from $F_0=A \otimes_k k \otimes_k A \cong A \otimes_k A$ to $A$ is the minimal graded free resolution of $A$ over $A^e$ \cite[Proposition 3.1]{vandenBergh}.  The complex is $\F^A$ is the totalization of the double complex with differentials $\partial'$ and $\partial''$ depicted in Figure \ref{fulldiagram}. In particular, $F_n$ is the sum of the modules on the $n$-th antidiagonal of this double complex.
  \end{defn}
  \vspace{-2em}
  \begin{figure}[h!]
\begin{equation*}
\begin{aligned} 
    \xymatrixrowsep{1pc} \xymatrixcolsep{2pc}
    \xymatrix{    
%%%%%%%%%%%%%%%%%%%%%%%%%%%%%%%%%%%%%%%%%%%%%%%%%%% row 0
\vdots\ar[d]_{\partial'} 
      & \vdots\ar[d]_{\partial'} 
      &\vdots\ar[d]_{\partial'} 
      &  
%      & \Lambda^{d-1} \otimes S_{i-1} \ar[d]_d
%      & \vdots\ar[d]_{\partial'} 
    %  & 
\\
%%%%%%%%%%%%%%%%%%%%%%%%%%%%%%%%%%%%%%%%%%%%%%%%%%% row 1
A \otimes_k (A^!)^*_a \otimes_k A_{0} \ar[d]_{\partial'} \ar[r]^-{(-1)^a\partial''}
      & A \otimes_k (A^!)^*_{a-1} \otimes_k A_1 \ar[d]_{\partial'} \ar[r]^-{(-1)^{a-1}\partial''}
      & A \otimes_k (A^!)^*_{a-2} \otimes_k A_{2} \ar[d]_{\partial'} \ar[r]^{\ \ \ \ \ (-1)^{a-2}\partial''}
      &  \cdots
 %\ar[r]^-{\partial''}
%      & \Lambda^2 \otimes S_{i-1} \ar[r]^{\kappa}\ar[d]_d 
      %& A \otimes_k (A^!)^*_1 \otimes_k A_{a-1} \ar[d]_{\partial'}
\\
%%%%%%%%%%%%%%%%%%%%%%%%%%%%%%%%%%%%%%%%%%%%%%%%%%% row 2
      A \otimes_k (A^!)^*_{a-1} \otimes_k A_0 \ar[d] \ar[r]^-{(-1)^{a-1}\partial''} 
      & A \otimes_k (A^!)^*_{a-2} \otimes_k A_{1} \ar[d] \ar[r]^-{(-1)^{a-2}\partial''} 
      & A \otimes_k (A^!)^*_{a-3} \otimes_k A_{2} \ar[d] \ar[r]^{\ \ \ \ \ (-1)^{a-3}\partial''} 
      & \cdots %\ar[r]^-{-\partial''}
%      & \Lambda^1 \otimes S_{i-1} \ar[r]^{\kappa}\ar[d]_d 
      %& A \otimes_k (A^!)^*_0 \otimes_k A_{a-1}
\\
%%%%%%%%%%%%%%%%%%%%%%%%%%%%%%%%%%%%%%%%%%%%%%%%%%% row 3
      \vdots \ar[d]_{\partial'} 
      & \vdots\ar[d]_{\partial'}
      & \vdots\ar[d]_{\partial'}
      & %\cdots \ar[r]^{\kappa}
%      & \Lambda^0 \otimes S_{i-1} 
      %&
\\
%%%%%%%%%%%%%%%%%%%%%%%%%%%%%%%%%%%%%%%%%%%%%%%%%%% row 4
        A \otimes_k (A^!)^*_2 \otimes_k A_0 \ar[r]^{\partial''}\ar[d]_{\partial'}
      & A \otimes_k (A^!)^*_1 \otimes_k A_1 \ar[r]^{-\partial''}\ar[d]_{\partial'}
      & A \otimes_k (A^!)^*_0 \otimes_k A_2
      & 
      & 
      %& 
\\
%%%%%%%%%%%%%%%%%%%%%%%%%%%%%%%%%%%%%%%%%%%%%%%%%%% row 5
        A \otimes_k (A^!)^*_1 \otimes_k A_0 \ar[r]^{-\partial''}\ar[d]_{\partial'}
      & A \otimes_k (A^!)^*_0 \otimes_k A_1 
      &  
      & 
      & 
      %& 
\\
%%%%%%%%%%%%%%%%%%%%%%%%%%%%%%%%%%%%%%%%%%%%%%%%%%% row 6
        A \otimes_k (A^!)^*_0 \otimes_k A_0
      & 
      &   
      & 
     & 
      %& 
\\
}
\end{aligned}  
  \end{equation*}
 \caption{The minimal resolution of a Koszul algebra $A$ over $A^e$.}
 \label{fulldiagram}
 \end{figure}
\begin{rem}
\label{rem:Priddy}
Here is the explicit connection with Priddy's resolution: tensoring \eqref{eq:Fcomplex} on the right over $A$ with $k$ gives Priddy's minimal resolution of $k$ as a left $A$-module in Definition~\ref{def:Priddycx}, also called the generalized Koszul resolution. Tensoring \eqref{eq:Fcomplex} on the left gives the minimal resolution of $k$ as a right $A$-module. 
\end{rem}

\begin{ipg}
Considering the graded strands of \eqref{eq:Fcomplex}, one can write this complex as a totalization of an anticommutative double complex, which we also call $\mathbb F^A$, of free $A$-modules given by the free $A$-modules
\begin{equation}
    \label{eq:bicomplex}
F_{ij}=A \otimes_k (A^!)^*_i \otimes_k A_j
%= (A^!)^*_i \otimes_k (A \otimes_k A_j)
%= (A^!)^*_i \otimes_k \widetilde{A}_j,
\end{equation}
where we are using the first tensor factor as ``coefficients'' and the maps $\partial'$ and $\pm\partial''$ of Definition \ref{def:Fcx} become the vertical and horizontal maps, respectively in the diagram (\ref{fulldiagram}). 
Note that $i$ is the homological degree in the complex $\mathbb F^A$. 
\end{ipg}

\begin{ipg}
\label{}
One can interpret the complex $\mathbb F^A$ in the language of the functors introduced in \ref{Koszulduality} as $\mathbb F^A=L(R(A))$, where $A$ is viewed as a complex concentrated in homological degree 0. The discussion in \ref{def:Fcx} shows there is a quasi-isomorphism $L(R(A))\simeq A$. This was previously shown in \cite[Thm 2.12.1]{BGS} and is a particular case of Remark \ref{viakoszulduality}. 
\end{ipg}

Next we apply the discussion from Section~\ref{sec-enveloping} to the $A$-module $A/\fm^a$ and arrange its resolution into a double complex similarly to the one shown in Figure \ref{fulldiagram} above. 

\begin{cor}
\label{truncateddiagramhere}
Totalization of the truncation of the double complex \eqref{eq:bicomplex} obtained by removing the columns with index $j\geq a\geq 1$ gives a graded $A$-free resolution 
\begin{equation}
    \label{eq:Xcx}
\X^A_a%A \otimes_k (A^!)^*\!\!. \otimes_k A/\fm^a 
= A \otimes_k (A^!)^* \otimes_k A_{\leq a-1}
\xra{\simeq} A/\fm^a.
\end{equation}
\end{cor}
\begin{proof}
Applying Proposition \ref{prop:resMoverA} by
tensoring the resolution of $A$ over $A^e$ on the right over $A$ with $A/\fm^a$ gives a graded $A$-free resolution
\[
\X^A_a=A \otimes_k (A^!)^* \otimes_k A/\fm^a
\xra{\simeq} A/\fm^a
\]
By means of the $k$-vector space identification
\[
A/\fm^a = A_{\leq a-1}
\]
the resolution becomes
\[
\X^A_a=A \otimes_k (A^!)^*\otimes_k A_{\leq a-1}
\xra{\simeq} A/\fm^a
\]
Viewing graded strands, one can write this as a totalization of an anticommutative double complex of free $A$-modules given by the terms
$F_{ij=}A \otimes_k (A^!)^*_i \otimes_k A_j$
%{\text{ \ denoted by \ }}
%\widetilde{(A^!)^*_i \otimes_k A_j}
with $i \geq 0$, $0 \leq j \leq a-1$ of \eqref{eq:bicomplex} with the differentials inherited from those described in Definition \ref{def:Fcx}. 
\end{proof}

 We display the diagram for the double complex that yields the $A$-free graded resolution $\X^A_a$ of $A/\fm^a$  in Corollary \ref{truncateddiagramhere}  in Figure \ref{fig:2}. In the language of \ref{Koszulduality} this resolution can be described as $\X^A_a=L(R(A/\m^a))$.

  \begin{figure}[h!]
  \begin{equation*}
\label{truncateddiagram}
\begin{aligned} 
    \xymatrixrowsep{1pc} \xymatrixcolsep{2pc}
    \xymatrix{    
%%%%%%%%%%%%%%%%%%%%%%%%%%%%%%%%%%%%%%%%%%%%%%%%%%% row 0
\vdots\ar[d]_{\partial'} 
      & \vdots\ar[d]_{\partial'} 
      %&\vdots\ar[d]_{\partial'} 
      &  
%      & \Lambda^{d-1} \otimes S_{i-1} \ar[d]_d
      & \vdots\ar[d]_{\partial'} 
      & 
\\
%%%%%%%%%%%%%%%%%%%%%%%%%%%%%%%%%%%%%%%%%%%%%%%%%%% row 1
A \otimes_k (A^!)^*_a \otimes_k A_{0} \ar[d]_{\partial'} \ar[r]^-{(-1)^a\partial''}
      & A \otimes_k (A^!)^*_{a-1} \otimes_k A_1 \ar[d]_{\partial'} \ar[r]^-{(-1)^{a-1}\partial''}
      %& A \otimes_k (A^!)^*_{a-2} \otimes_k A_{2} \ar[d]_{\partial'} \ar[r]^{\ \ \ \ \ (-1)^{a-2}\partial''}
      &  \cdots
 \ar[r]^-{\partial''}
%      & \Lambda^2 \otimes S_{i-1} \ar[r]^{\kappa}\ar[d]_d 
      & A \otimes_k (A^!)^*_1 \otimes_k A_{a-1} \ar[d]_{\partial'}
\\
%%%%%%%%%%%%%%%%%%%%%%%%%%%%%%%%%%%%%%%%%%%%%%%%%%% row 2
      A \otimes_k (A^!)^*_{a-1} \otimes_k A_0 \ar[d] \ar[r]^-{(-1)^{a-1}\partial''} 
      & A \otimes_k (A^!)^*_{a-2} \otimes_k A_{1} \ar[d] \ar[r]^-{(-1)^{a-2}\partial''} 
      %& A \otimes_k (A^!)^*_{a-3} \otimes_k A_{2} \ar[d] \ar[r]^{\ \ \ \ \ (-1)^{a-3}\partial''} 
      & \cdots \ar[r]^-{-\partial''}
%      & \Lambda^1 \otimes S_{i-1} \ar[r]^{\kappa}\ar[d]_d 
      & A \otimes_k (A^!)^*_0 \otimes_k A_{a-1}
\\
%%%%%%%%%%%%%%%%%%%%%%%%%%%%%%%%%%%%%%%%%%%%%%%%%%% row 3
      \vdots \ar[d]_{\partial'} 
      & \vdots\ar[d]_{\partial'}
      %& \vdots\ar[d]_{\partial'}
      & %\cdots \ar[r]^{\kappa}
%      & \Lambda^0 \otimes S_{i-1} 
      &
\\
%%%%%%%%%%%%%%%%%%%%%%%%%%%%%%%%%%%%%%%%%%%%%%%%%%% row 4
       % A \otimes_k (A^!)^*_2 \otimes_k A_0 \ar[r]^{\partial''}\ar[d]_{\partial'}
     % & A \otimes_k (A^!)^*_1 \otimes_k A_1 \ar[r]^{-\partial''}\ar[d]_{\partial'}
      %& A \otimes_k (A^!)^*_0 \otimes_k A_2
     % & 
%      & 
    %  & 
%\\
%%%%%%%%%%%%%%%%%%%%%%%%%%%%%%%%%%%%%%%%%%%%%%%%%%% row 5
        A \otimes_k (A^!)^*_1 \otimes_k A_0 \ar[r]^{-\partial''}\ar[d]_{\partial'}
      & A \otimes_k (A^!)^*_0 \otimes_k A_1 
      &  
      & 
%      & 
      & 
\\
%%%%%%%%%%%%%%%%%%%%%%%%%%%%%%%%%%%%%%%%%%%%%%%%%%% row 6
        A \otimes_k (A^!)^*_0 \otimes_k A_0
      & 
      &   
      & 
%      & 
      & 
\\
}
\end{aligned}  
  \end{equation*}
  \caption{The the $A$-free graded resolution $\X^A_a$ of $A/\fm^a$.}
  \label{fig:2}
 \end{figure}

%\begin{rem}
%\label{rem:augmentationmap}
%We describe the sign of the augmentation map from the double complex to $A$. The maps $\epsilon_i$ to $A$ must be given signs such that the following diagram anticommutes for each $i \ge 0$:
%\[
% \xymatrixrowsep{1.5pc} \xymatrixcolsep{3.5pc}\xymatrix{
%  \widetilde{(A^\dperp)_1^* \otimes A_i} \ar[r]^{-\partial''}\ar[d]^{\partial'}
%&      \widetilde{(A^\dperp)_0^* \otimes  A_{i+1}} \ar[d]^{\epsilon_{i+1}}
%\\
% \widetilde{(A^\dperp)_0^*  \otimes A_i} \ar[r]^{\epsilon_i}
%& A
%}
%\]
%We will set $\epsilon_i$ to be the multiplication map (see Definition \ref{def:doublecx}) with a positive sign for all $i \ge 0$.
%\end{rem}\cm

\begin{rem}
The resolution $\X^A_a$ is minimal for $a=1$  in which case $\X^A_a$ recovers the Priddy complex without its first term.  However for $a\geq 2$ this resolution is typically non minimal as the rows are split acyclic; see \ref{chunk-L}. The goal of Section~\ref{sec-minimal} is to produce a minimal free resolution for $A/\m^a$ using $\X^A_a$.
\end{rem}

%%%%%%%%%%%%%%%%%%%%%%%%%
\section{Minimal resolutions for powers of the maximal ideal}
\label{sec-minimal}

We introduce  complexes $\mathbb L^A_a$ inspired by work of  Buchsbaum and Eisenbud \cite{BE}. These will turn out to be the minimal resolutions for the powers of the homogeneous maximal ideal of a graded Koszul algebra.

\begin{ipg} 
\label{chunk-L}
We define free $A$-modules analogous to the Schur modules used by Buchsbaum and Eisenbud in their resolutions (the case where $A$ is a polynomial ring) in \cite{BE}. 
First note that the rows of the double complex \eqref{eq:bicomplex} except the bottom one are exact; in fact, they can be viewed as the result of applying the exact base change $A\otimes_k -$ to the strands of the dual Priddy complex, all of which are exact except the one whose homology is $k$ (in the case where $A$ is a polynomial ring, it is applied to the strands of the tautological Koszul complex; see \cite[Section 1.4]{MiRa-18}). These complexes are contractible, as they consist of free $A$-modules, and so all kernels, images, and cokernels of the differentials are free as well. 

Define for $a>0$ the following free $A$-modules
\begin{align}
\label{Ldef}
L^A_{n, a}
&=\im\left(A\otimes_k {A^{!}}^*_{n+1}\otimes_k A_{a-1} %\stackrel{(-1)^{n+1}\partial''}{\longrightarrow} 
%EF: changed arrows to \xrightarrow{} so that they have the right length
\xrightarrow{(-1)^{n+1}\partial''}
A\otimes_k {A^{!}}^*_n\otimes_kA_{a} \right)
\\
&= \ker \left(A\otimes_k {A^{!}}^*_{n}\otimes_k A_{a} \xrightarrow{(-1)^{n}\partial''}
%\stackrel{(-1)^{n}\partial''}{\longrightarrow} 
A\otimes_k {A^{!}}^*_{n-1}\otimes_kA_{a+1} \right) \ .
\end{align}

The vertical differentials $\partial'$ in the diagram \ref{truncateddiagram} induce maps on these modules, which we again denote by $\partial'$, to yield a complex 
\begin{equation}
\label{Lcomplex}
    \mathbb L^A_a \colon \  
    \cdots \to L^A_{n,a}
    \xra{\partial_{n}'} L^A_{n-1,a}
    \xra{\partial_{n-1}'} \dots \xra{\partial_{1}'}
L^A_{0,a}.
\end{equation}
This complex is minimal in the sense that $\partial'(L^A_{n,a})\subseteq \m L^A_{n-1,a}$ for all $n$ since the same property holds for the columns of the complex in Figure \ref{fulldiagram}
viewed as complexes with differential $\partial'$.  The construction of the complex $\mathbb L^A_a$ is depicted in Figure \ref{diagramwithLs}.
\end{ipg} 

%\vspace{-2em}
%{\crd to make narrower, can take out $k$ on tensor products or use scaling}
\begin{figure}[h!]
 \begin{equation*}
%\resizebox{1.1\textwidth}{!}{
\begin{aligned} 
    \xymatrixrowsep{1pc} \xymatrixcolsep{0.5pc}
    \xymatrix{   
%%%%%%%%%%%%%%%%%%%%%%%%%%%%%%%%%%%%%%%%%%%%%%%%%%%%%row -2
 \vdots\ar[d]_{\partial'} 
  & \vdots\ar[d]_{\partial'} 
 &
  &\vdots\ar[d]_{\partial'} 
  & \vdots\ar[d]_{\partial'_{n+1}} 
  & 
  &
 \\
%%%%%%%%%%%%%%%%%%%%%%%%%%%%%%%%%%%%%%%%%%%%%%%%%%%%%row -1
A \otimes_k (A^!)^*_{n+a} \otimes_k A_{0} \ar[d]_{\partial'} \ar[r]
      & A \otimes_k (A^!)^*_{n+a-1} \otimes_k A_1 \ar[d]_{\partial'} \ar[r]
      %& A \otimes_k (A^!)^*_{a-2} \otimes_k A_{2} \ar[d]_{\partial'} \ar[r]^{\ \ \ \ \ (-1)^{a-2}\partial''}
      &  \cdots
 \ar[r]
%      & \Lambda^2 \otimes S_{i-1} \ar[r]^{\kappa}\ar[d]_d 
      & A \otimes_k (A^!)^*_{n+1} \otimes_k A_{a-1} \ar[d]_{\partial'}
      \ar[r]
      &L_{n,a}^A \ar[d]_{\partial'_n}
\\
%%%%%%%%%%%%%%%%%%%%%%%%%%%%%%%%%%%%%%%%%%%%%%%%%%% row 0
 \vdots\ar[d]_{\partial'} 
       & \vdots\ar[d]_{\partial'} 
       &
       &\vdots\ar[d]_{\partial'} 
       & \vdots\ar[d]_{\partial'_1} 
       & 
       &
 \\
%%%%%%%%%%%%%%%%%%%%%%%%%%%%%%%%%%%%%%%%%%%%%%%%%%% row 1
A \otimes_k (A^!)^*_a \otimes_k A_{0} \ar[d]_{\partial'} \ar[r]
      & A \otimes_k (A^!)^*_{a-1} \otimes_k A_1 \ar[d]_{\partial'} \ar[r]
      %& A \otimes_k (A^!)^*_{a-2} \otimes_k A_{2} \ar[d]_{\partial'} \ar[r]^{\ \ \ \ \ (-1)^{a-2}\partial''}
      &  \cdots
 \ar[r]
%      & \Lambda^2 \otimes S_{i-1} \ar[r]^{\kappa}\ar[d]_d 
      & A \otimes_k (A^!)^*_1 \otimes_k A_{a-1} \ar[d]_{\partial'}
      \ar[r]
      &L_{0,a}^A 
\\
%%%%%%%%%%%%%%%%%%%%%%%%%%%%%%%%%%%%%%%%%%%%%%%%%%% row 2
      A \otimes_k (A^!)^*_{a-1} \otimes_k A_0 \ar[d] \ar[r] 
      & A \otimes_k (A^!)^*_{a-2} \otimes_k A_{1} \ar[d] \ar[r] 
      %& A \otimes_k (A^!)^*_{a-3} \otimes_k A_{2} \ar[d] \ar[r]^{\ \ \ \ \ (-1)^{a-3}\partial''} 
      & \cdots \ar[r]
%      & \Lambda^1 \otimes S_{i-1} \ar[r]^{\kappa}\ar[d]_d 
      & A \otimes_k (A^!)^*_0 \otimes_k A_{a-1} 
      &
\\
%%%%%%%%%%%%%%%%%%%%%%%%%%%%%%%%%%%%%%%%%%%%%%%%%%% row 3
      \vdots \ar[d]_{\partial'} 
      & \vdots\ar[d]_{\partial'}
      %& \vdots\ar[d]_{\partial'}
      & %\cdots \ar[r]^{\kappa}
%      & \Lambda^0 \otimes S_{i-1} 
      &
      &
\\
%%%%%%%%%%%%%%%%%%%%%%%%%%%%%%%%%%%%%%%%%%%%%%%%%%% row 4
       % A \otimes_k (A^!)^*_2 \otimes_k A_0 \ar[r]^{\partial''}\ar[d]_{\partial'}
     % & A \otimes_k (A^!)^*_1 \otimes_k A_1 \ar[r]^{-\partial''}\ar[d]_{\partial'}
      %& A \otimes_k (A^!)^*_0 \otimes_k A_2
     % & 
%      & 
    %  & 
%\\
%%%%%%%%%%%%%%%%%%%%%%%%%%%%%%%%%%%%%%%%%%%%%%%%%%% row 5
        A \otimes_k (A^!)^*_1 \otimes_k A_0 \ar[r]\ar[d]_{\partial'}
      & A \otimes_k (A^!)^*_0 \otimes_k A_1 
      &  
      & 
      & 
      & 
\\
%%%%%%%%%%%%%%%%%%%%%%%%%%%%%%%%%%%%%%%%%%%%%%%%%%% row 6
        A \otimes_k (A^!)^*_0 \otimes_k A_0
      & 
      &   
      & 
      & 
      & 
\\
}
\end{aligned}  
%}
  \end{equation*}
  \caption{The construction of the complex $\mathbb L^A_a$.}
  \label{diagramwithLs}
 \end{figure}

\begin{lemma}
\label{lem:augmentation}
The complex $\L^A_a$ can be augmented by the evaluation map 
\begin{equation}
\label{epsilon}
\varepsilon_a\colon L^A_{0,a}=A\otimes_k {A^{!}}^*_{0} \otimes_k A_a\to \m^a
\end{equation}
which is the restriction of the multiplication map 
\[ 
\varepsilon:A\otimes_k {A^{!}}^*_{0} \otimes_k A\to A \text{ sending  } r\otimes v\otimes s\mapsto rvs.
\]
\end{lemma}
\begin{proof}
As stated in Definition \ref{def:Fcx}, $\varepsilon$ is an augmentation map $\F^A_\bullet\to A$. We verify explicitly that $\varepsilon$ satisfies the required property $\varepsilon\circ(\partial'-\partial'')=0$ below:
\begin{eqnarray*}
\varepsilon\circ(\partial'-\partial'')(r\otimes v\otimes s)=\varepsilon(rv\otimes 1\otimes s-r\otimes 1\otimes vs)=rvs-rvs=0.
\end{eqnarray*}
Since $\partial''|_{L^A_{1,a}}=0$ it follows from the computation above that $\varepsilon\circ \partial' (L^A_{1,a})=0$, hence the complex $  \mathbb L^A_a$ can be augmented to 
\begin{equation*}
    \cdots \to L^A_{n,a}\xra{\partial_{n}'}
L^A_{n-1,a}\xra{\partial_{n-1}'} \dots \xra{\partial_{1}'}
L^A_{0,a} \xra{\varepsilon_a}\m^a\to 0.
\end{equation*}
\end{proof}

The following is the main result of our paper. The case when $A$ is an exterior algebra has appeared previously in \cite[Corollary 5.3]{EFS}. The proof therein uses the self-injectivity of the exterior algebra in a crucial manner, and therefore does not seem to extend to all Koszul algebras. 

\begin{thm}
\label{thm:minimalres}
If $A$ is a Koszul algebra, the complexes $\mathbb L^A_a$ defined in equation \eqref{Lcomplex} with the augmentation map $\varepsilon_a$ defined in \eqref{epsilon} are minimal free resolutions for the powers $\m^a$ of the maximal ideal with $a\geq 1$.
\end{thm}

\begin{proof}
The complexes $\L^A_a$ are minimal by the discussion in \ref{chunk-L}. The proof of the remaining claims is by induction on $a\geq 1$. 

The definition of $L_{n,1}$ in \eqref{Ldef} shows that there are isomorphisms 
\[
L^A_{n,1}\cong A\otimes_k {A^!}^*_{n+1} \otimes_k k \cong  A\otimes_k {A^!}^*_{n+1},
\]
since the map $\partial ''$ is injective on $\mathbb F^A_{n+1,0}$ for $n\geq 0$, the leftmost column of the double complex in Figure \ref{fulldiagram}
%\eqref{eq:bicomplex} 
(this column considered by itself is in fact $\X^A_1$). Therefore there is an isomorphism of complexes $\mathbb L^A_1\cong (\X^A_1)_{\geq 1}[-1]$, where $(\X^A_1)_{\geq 1}$ denotes the truncation of the complex $\X^A_1$ by removing the homological degree 0 component. Since $\X^A_1=P^A_\bullet$ is just the Priddy complex (upon noting that $\partial' \colon (\X^A_1)_1 \to (\X^A_1)_0$ agrees with $\varepsilon$ under the identification $(\X^A_1)_0 = A\otimes_k {A^!}^*_0 \otimes_k A_0 \cong A$), we see that  
$(\X^A_1)_{\geq 1}\xra{\varepsilon} \m$ is a resolution of $\m$ by \ref{Priddyres} and the base case that $\mathbb L^A_1 \xra{\varepsilon} \m$ is a minimal resolution of $\m$ follows.

For arbitrary $a\geq 2$, \eqref{Ldef} gives a short exact sequence of complexes
\[
0\to \L^A_{a-1}\to P^A_\bullet \otimes_k A_{a-1} \to \L^A_a[-1] \to 0,
\]
where $P^A_n \otimes_k A_{a-1}=A\otimes_k {A^!}^*_n\otimes_k A_{a-1}$ is the $(a-1)$-st column of the double complex \eqref{truncateddiagram}. The notation signifies that this column can be viewed as the Priddy complex $P^A_\bullet$ tensored with $A_{a-1}$. From the long exact sequence in homology induced by the short exact sequence of complexes displayed above we deduce 
\[
\HH_i(\L^A_a)=\begin{cases}
0 & i\geq 1\\
\ker\left(\HH_0(\L^A_{a-1})\to \HH_0(P^A_\bullet\otimes_k A_{a-1}) \right) & i=0.
\end{cases}
\] 
It remains to show that $\HH_0(\L^A_a)=\m^a$. Indeed, the induced map in homology 
\[ 
\HH_0(\L^A_{a-1})\to \HH_0(P^A_\bullet\otimes_k A_{a-1})\cong \HH_0(P^A_\bullet)\otimes_k A_{a-1}
\]
can be recovered as the bottom map in the following commutative diagram
\[
 \xymatrixrowsep{1.5pc} \xymatrixcolsep{3.5pc}
 \xymatrix{
 L^A_{0,a-1} \ar@{=}[r]\ar[d]^{\varepsilon_{a-1}}
 % \widetilde{(A^\dperp)_1^* \otimes A_i} \ar[r]^{-\partial''}\ar[d]^{\partial'}
&     A\otimes_k(A^\dperp)_0^* \otimes_k  A_{a-1} \ar[d]^{\varepsilon_0\otimes \id_{A_{a-1}}}
\\
\m^{a-1} \ar[r]^{} & k\otimes_k A_{a-1}\cong \m^{a-1}/\m^{a}.
}
\]
Commutativity of the diagram  yields that for $r\otimes s\otimes v\in  \L^A_{0,a}$ the induced map in homology is given by  
\[\varepsilon_a(r\otimes s\otimes v)=rsv \mapsto \overline{r}sv,\]
where $\overline{r}$ is the coset of $r$ in $k=A/\m$.
Thus we obtain the desired identification 
\begin{eqnarray*}
&&\ker\left(\HH_0(\L^A_{a-1})\to \HH_0(\PP^A_\bullet\otimes_k A_{a-1}) \right)\\
&=&\varepsilon_a\left(\Span\{r\otimes s\otimes v \mid r\in \m, s\in k, v\in A_{a-1}\}\right)\\
 &\cong& \fm^{a}.
\end{eqnarray*}
\end{proof}

\begin{rem}
\label{rem:alternate}
Recall that rows of the double complex \eqref{eq:bicomplex} except the bottom one are exact; in fact, they can be viewed as the result of applying a base change to the strands of the dual Priddy complex; see \ref{chunk-L}. These complexes are contractible, as they consist of free $R$-modules. Hence for $n\geq 1$ the $(n+a)$-th row of \eqref{truncateddiagram}, counting from the bottom (as the 0th row), is quasi-isomorphic to $L_{n,a}^A$ and the lower rows (numbered 1 through $a$) are split exact. The acyclic assembly lemma \cite[Lemma 2.7.3]{Weibel} yields quasi-isomorphisms $(\X^A_a)_{\geq 1}[-1] \xra{\simeq} \L^A_a$ for $a\geq 1$.  As shown in Corollary \ref{truncateddiagramhere} there are quasi-isomorphisms $\X^A_a \xra{\simeq} A/\m^a$, hence also $(\X^A_a)_{\geq 1}\xra{\simeq} \m^a$. Transitivity yields a quasi-isomorphism $\L^A_a \xra{\simeq} \m^a$.

This approach gives an alternate proof for our main result, but it only determines the augmentation map up to an isomorphism on its target. We prefer the more explicit approach of Theorem \ref{thm:minimalres}, which specifies the augmentation map $\varepsilon_a$.
\end{rem}

The following corollary of Theorem \ref{thm:minimalres} gives an explicit formula for the Betti numbers of powers of the maximal ideal of a Koszul algebra. That $\m^a$ has an $a$-linear minimal free resolution also follows from \cite[Theorem 3.2]{Sega}. Once the linearity of this resolution has been established, \cite[Chapter 2, Corollary 3.2 (iiiM)]{Polishchuk} gives an alternate interpretation for the Betti numbers of $\m^a$ in terms of the graded components of a quadratic dual module for the $A$-module $\m^a$. However this description seems less amenable to explicit computations than our methods.

The next result utilizes the  Hilbert series 
\[
H_{(A^!)^*}(t)=\sum_{\ell\geq 0} \dim_k (A^{!})^*_\ell \cdot t^\ell \text{ and } H_{A/\m^a}(t)=\sum_{j= 0}^{a-1} \dim_k A_j \cdot t^j.
\] 

\begin{cor}
\label{cor:betti}
If $(A,\m)$ is a Koszul algebra, the nonzero graded Betti numbers of the powers of $\m$ are given by
\[
\beta_{n,n+a}^A(\m^a)=\sum_{i=1}^a (-1)^{i+1} \dim_k( (A^!)^*_{n+i})\dim_k(A_{a-i}).
\]
In particular, the minimal graded resolution of $\m^a$ is $a$-linear and its graded Poincar\'e series is
\[
P^A_{\m^a}(y,z)=-(-z)^{-a}H_{(A^!)^*}(yz) H_{A/\m^a}(-yz).
\]

Consequently, the nonzero Betti numbers of $A/\m^a$ are given by
\[
\beta_{n,j}^A(A/\fm^a)=\begin{cases}
\sum_{i=1}^{a} (-1)^{i+1} \dim_k( (A^!)^*_{n+i-1})\dim_k(A_{a-i}) & n > 0, \ j=n+a-1 \\
1 & n=j=0.
\end{cases}
\]
\end{cor}

\begin{proof}
The fact that minimal free resolution of $\m^a$ is $a$-linear follows from the Theorem \ref{thm:minimalres} and the description of the differential $\partial'$ of the complex \eqref{Lcomplex} in view of the fact that there is a splitting of the map $\partial''$ to each $L_{n,a}$ identifying a basis of it with part of a basis of the last column of $\mathbb X^A_a$. Consider the rows of the truncated complex $\X^A_a$ when augmented to the relevant $L_{n,a}$ as follows. 
\[
0
\lra
A \otimes_k (A^!)^*_{n+a} \otimes_k A_{0} \lra
\cdots
\lra
A \otimes_k (A^!)^*_{n+1} \otimes_k A_{a-1} 
\lra
L^A_{n,a}
\lra 0
\]
The exactness of this complex, as explained in Remark \ref{rem:alternate}, yields the identities
\begin{eqnarray*}
\beta_{n,n+a}^A(\m^a) 
&=& \rank_A( L^A_{n,a})=\sum_{i=1}^a (-1)^{i+1} \rank_A \left( {A\otimes_k (A^!)^*_{n+i} \otimes_k A_{a-i}} \right)
\\
&=&\sum_{i=1}^a (-1)^{i+1} \dim_k (A^!)^*_{n+i}\dim_k A_{a-i}
\end{eqnarray*}
and the vanishing of the remaining Betti numbers is due to the fact that the minimal resolution in Theorem \ref{thm:minimalres} is $a$-linear. 

 Lastly, the coefficients of the series
\[
\frac{-H_{(A^!)^*}(yz) H_{A/\m^a}(-yz)}{(-z)^a}=\sum_{n\geq 0} \left( \sum_{\substack{0\leq j\leq a-1 \\j+\ell=n+a}}(-1)^{j-a}\dim_k(A^{!})^*_\ell \dim_k A_j \right) z^ny^{n+a} ,
\]
agree with the preceding expression for $\beta^A_{n,n+a}(\m^a)$ by setting $j=a-i, \ell=n+i$.
\end{proof}

% \begin{rem}
% \label{rem-flatM}
% The results above can actually be extended to obtain the minimal graded resolution of any nonzero Artinian graded $A$-module $N$ that is {\it flat} as an $A/\fm^aA$-module for some $a>0$. This includes the case that $N=M/\fm^aM$ for some flat $A$-module $M$. 

% Indeed, let 
% \[
% a=1+\max\{i | N_i \neq 0  \}
% \]
% so that $N$ is also an $A/\fm^aA$-module. 
% Notice that this is the only value for which $N$ could even be flat over $A/\fm^aA$ (it is a module over $A/\fm^bA$ for all $b\geq a$, but has a nontrivial annihilator, namely $\fm^b/\fm^a$). 
% \end{rem}
%%%%%%%%%%%%%%%%%%%%%%%%%

In contrast to Theorem \ref{thm:minimalres}, for non Koszul algebras the $A$-free resolution of $A/\fm^a$ afforded by Corollary \ref{truncateddiagramhere} cannot be minimized by the procedure presented in this section. We illustrate the obstructions by means of the following example.

\begin{example}
\label{ex-nonkoszul}
Let $A=k[x]/(x^3)$, which is a non Koszul (also non quadratic) algebra. The enveloping algebra is 
\[
A^e=A\otimes_kA =k[x]/(x^3)\otimes_k k[y]/(y^3)\cong k[x,y]/(x^3,y^3)
\]
and the $A^e$-module structure induced on $A$ by the (surjective) multiplication map $A^e=A\otimes_kA\xra{\varepsilon} A$ yields the isomorphism $A\cong A^e/(x-y)$. Therefore $A$ has the following two-periodic resolution over the complete intersection $A^e$
\[
\cdots \ra A^e \xra{x-y} A^e \xra{x^2+xy+y^2} A^e \xra{x-y} A^e \xra{\varepsilon} A\ra 0.
\]
Rewriting this complex in the form of Section~\ref{sec-enveloping} gives
\[
\cdots \ra A\otimes_k V_2 \otimes_k A \xra{\partial} A\otimes_k V_1 \otimes_k A \xra{\partial} A\otimes_k V_0 \otimes_k A  \xra{\varepsilon} A\ra 0,
\]
where each $V_i$ is a one dimensional vector space with basis $\{e_i\}$ and for $i>0$
\[
\partial(1\otimes e_i \otimes 1 )=
\begin{cases}
x\otimes e_{i-1}\otimes 1-1\otimes e_{i-1}\otimes y & \text{ for } i \text{ odd}\\
x^2\otimes e_{i-1}\otimes 1 + x\otimes e_{i-1}\otimes y +1\otimes e_{i-1}\otimes y^2 & \text{ for } i \text{ even}.
\end{cases}
\]

The conclusion of Corollary \ref{truncateddiagramhere} still holds and indicates that the truncated complexes $\X^A_a$ are (non minimal) free resolutions for $A/\fm^a$. But by contrast to the Koszul case, we see that arranging by grading as in \eqref{eq:bicomplex} yields a diagram that is not a bicomplex and whose rows are no longer exact (or even complexes!), and so in the truncated complex \eqref{eq:Xcx} the rows are no longer acyclic. Correspondingly, the modules $L_{n,a}$ one could define are no longer free. 
Thus there is no clear way to minimize the complex $\X^A_a$ in a similar manner to the technique used in this section, except for the case $a=1$ where $\X^A_a$ is already minimal.
\end{example}

%%%%%%%%%%%%%%%%%%%%%%%%%
\section{Examples}
\label{sec-examples}
%%%%%%%%%%%%%%%%%%%%%%%%%

In this section we provide examples which illustrate our constructions for certain Koszul algebras. For simplicity, all our examples are commutative algebras defined by quadratic monomial ideals, but of course there are plenty of noncommutative examples as well. This class  is known to yield Koszul algebras by \cite{Froberg}.

\begin{example}
\label{ex:1}
Consider the following pair of dual Koszul algebras from Example \ref{ex:0}
\[
A =\frac{ k[x,y,z]}{(x^2,xy,y^2)}
%=\frac{k\langle x,y,z \rangle}{(x^2,xy,y^2, xz-zx, xy-yx, yz-zy)}=\frac{k\langle x,y,z \rangle}{Q}.\]
\qquad \text{and} \qquad
A^!=\frac{k\langle x^*, y^*, z^*\rangle}{((z^*)^2, x^*z^*+z^*x^*, y^*z^*+z^*y^*)}.
\]
The graded pieces $(A^!)_{-n}$ are spanned by the words of length $n$ on the alphabet $\{x^*,y^*,z^*\}$ where the first letter is $x^*, y^*$ or $z^*$ and the other $n-1$ are $x^*$ or $y^*$, whence $\dim_k (A^!)_{-n} = 3\cdot 2^{n-1}$ for $n\geq 1$. For $n\geq 1$, $A_n$ is spanned by monomials of the form $(z^*)^n$, $x^*(z^*)^{n-1}$, and $y^*(z^*)^{n-1}$ so that $\dim_k A_{-n}=3$. Thus in this case both $A$ and $A^!$ are infinite dimensional $k$-algebras.

The Priddy complex $P^A_\bullet$ \eqref{def:Priddycx} consists of terms of the form 
\begin{eqnarray*}
P_0 &=&A \otimes_k k \\ 
P_n &=& A \otimes_k k^{3\cdot 2^{n-1}}\text{ for } n\geq 1
\end{eqnarray*}
%Thus the Priddy complex \eqref{def:Priddycx}, yields the Betti numbers  $\beta^0_A(k)=1$ and $\beta^n_A(k)=3\cdot 2^{n-1}$ for $n\geq 1$ and it can be verified directly that the Poincar\'e series is the reciprocal of the Hilbert series of $A$ evaluated at $-t$:
%\[
%P^A_k(t)=1+\sum_{n\geq 1}3\cdot 2^{n-1}t^n =1+\frac{3t}{1-2t}=\frac{1+t}{1-2t}=\frac{1}{H_A(-t)}.
%\]
and the resolution of $A$ over $A^e$ viewed as a double complex \eqref{eq:bicomplex} has terms
\[
F_{i,j}=
%\begin{cases}
%A \otimes_k k\otimes_k A_j & i=0,\\
A \otimes_k k^{3\cdot 2^{i-1}} \otimes_k A_j
%\end{cases}
=
\begin{cases}
%A\otimes _A A & i= j=0,\\
%A\otimes _A A^3 & i= 0, j>0\\
A^{3\cdot 2^{i-1}}\otimes_A A& i\geq 1, j=0\\
A^{3\cdot 2^{i-1}}\otimes_A A^3& i\geq 1, j\geq 1.
\end{cases}
\]
Corollary \ref{cor:betti} reveals that the Betti numbers of $\fm^a$ are independent of $a$. Indeed for $a\geq 1$ and $n\geq 0$ we have
\begin{eqnarray*}
\beta_{n,n+a}(\fm^a)& =& \sum_{i=1}^{a-1} (-1)^{i+1}\cdot9 \cdot 2^{n+i-1}+(-1)^{a+1}\cdot3 \cdot 2^{n+a-1}\\
&=&  9\cdot 2^{n}\cdot  \frac{1-(-2)^{a-1}}{3} +(-1)^{a-1} \cdot3 \cdot 2^{a+n-1} \\
&=&3\cdot 2^n .
\end{eqnarray*}
\end{example}

%\begin{example}
%Consider the following complete intersection Koszul algebra  
%\[A =\frac{ k[x,y,z]}{(x^2, y^2)}=\frac{k\langle x,y,z \rangle}{(x^2,y^2, xz-zx, xy-yx, yz-zy)}.\]
%%with Hilbert series $H_A(t)=\frac{(1-t^2)^2}{(1-t)^3}=\frac{t^2+2t+1}{1-t}$.
%The Koszul dual algebra is given by 
%\[A^!=\frac{k\langle x^*, y^*, z^*\rangle}{((z^*)^2, x^*y^*+y^*x^*, x^*z^*+z^*x^*, y^*z^*+z^*y^*)}.\]
%In this case we find that 
%\[
%F_{i,j}=A \otimes_k {A^!}^*_i\otimes_k A_j =
%\begin{cases}
%A^{2i+1}, & j=0 \\
%A^{3(2i+1)}, & j=1 \\
%A^{4(2i+1)}, & j\geq 2 
%\end{cases}
%\]
%and therefore the ranks of the modules $L_{n,a}$ can be computed by noting that
%\[
%\rank L_{n,2}=- \rank F_{n+2,0}+\rank F_{n+1,1}=4(n-1) 
%\]
%and for $a\geq 2$ we have
%\begin{eqnarray*}
%&& \rank L_{n,a+1}-\rank L_{n,a} \\
%&&\quad = \sum_{i=1}^{a+2} (-1)^{i+1} \rank F_{n+i,a+1-i}-\sum_{i=1}^{a} (-1)^{i+1} \rank F_{n+i,a-i}\\
%&&\quad =\pm\left[ \rank F_{a+n+1,0} - \rank F_{a+n,1}+ \rank F_{a+n-1,2} - F_{a+n,0} - \rank F_{a+n-1,1}\right]\\
%%&&=\pm\left[ 2(a+n+1)+1-3(2(a+n)+1)+4(2(a+n-2)+1)+2(a+n)+1-3(2(a+n)+1)\right]=0\\
%&& \quad =0.
%\end{eqnarray*}
%Thus the Betti numbers of $\fm^a$ are independent of $a$ for $a\geq 2$ and given by
%\begin{eqnarray*}
%\beta^A_{n,n+1}(\fm) &=& \rank L_{n,1}=\rank F_{n+1,0}=2n+3\\
%\beta^A_{n,n+a}(\fm^a) &=& \rank L_{n,a}=4n+4
%\quad  \text{ for } a\geq 2.
%\end{eqnarray*}
%\end{example}

\begin{example}
Consider the following commutative Koszul algebra  
\[A =\frac{ k[x,y,z]}{(xy,xz)}=\frac{k\langle x,y,z \rangle}{(xy,xz, xz-zx, xy-yx, yz-zy)}.\]
The Koszul dual algebra is given by 
\[A^!=\frac{k\langle x^*, y^*, z^*\rangle}{((x^*)^2, (y^*)^2, (z^*)^2, y^*z^*+z^*y^*)}\]
and its Hilbert function satisfies the Fibonacci recurrence
\[
\dim A^!_{-n-2}=\dim A^!_{-n}+\dim A^!_{-n-1}.
\]
Indeed, setting $u(n)$ to be the number of monomials in $A^!$ of degree $-n$ ending in $x$ and $v(n)$ to be the number of monomials in $A^!$ of degree $-n$ not ending in $x^*$, yields $u(n)=v(n-1)$ and $v(n)=2u(n-1)+u(n-2)$. The second expression follows because the number of  monomials ending in $y^*$ or $z^*$ where the previous letter is $x^*$ is $2u(n-1)$ and the number of monomials ending in $y^*z^*$ (or equivalently, $z^*y^*$) where the previous letter is $x$ is $u(n-2)$.  Thus this leads to
\begin{eqnarray*}
\dim A^!_{-n-2}&=&u(n+2)+v(n+2)=v(n+1)+2u(n+1)+u(n)\\
&=&v(n+1)+u(n+1)+v(n)+u(n)=\dim A^!_{-n-1} +\dim A^!_{-n}.
\end{eqnarray*}
This shows that the Betti numbers of $\fm$ are the Fibonnacci numbers starting with  $\beta^A_{0}(\fm)=3$ and $\beta^A_{1}(\fm)=5$.

The identity above in turn implies that the terms of the double complex as well as the free modules in the resolution of $\fm^a$ satisfy similar recurrences
\[
\rank F_{n+2,a}=\rank F_{n+1,a}+\rank F_{n,a}, \quad  \rank L_{n+2,a} =\rank L_{n+1,a}+\rank L_{n,a}.
\]
We conclude that the Fibonacci recurrence holds for Betti numbers 
\[
\beta^A_{n+2,n+2+a}(\fm^a) = \beta^A_{n+1,n+1+a}(\fm^a)+\beta^A_{n,n+a}(\fm^a) \text{ for } a\geq 1, n\geq 0
\]
subject,  if $a\geq 2$,  to the initial conditions $\beta^A_{0}(\fm^a)=a+4$ and $\beta^A_{1}(\fm^a)=2a+4$. Solving the above recurrence yields closed formulas for the Betti numbers of $\fm^a$ with  $a\geq 2$ as follows
\[
\beta^A_{n,n+a}(\fm^a) =
\left(\frac{a+4}{2}+\frac{3a+4}{2\sqrt{5}}\right)\left(\frac{1+\sqrt{5}}{2}\right)^n+\left(\frac{a+4}{2}-\frac{3a+4}{2\sqrt{5}}\right)\left(\frac{1-\sqrt{5}}{2}\right)^n.
\]
\end{example}

We now give an infinite resolution counterpart to a family of square-free monomial ideals that have appeared as ideals of the polynomial ring in work of Galetto \cite{Galetto}.

\begin{example}[See also {\cite[Example 3.18]{vandebogert}}]
Consider the dual pair of Koszul algebras 
\[
A=\frac{k[x_1,\ldots, x_d]}{(x_1^2,\ldots, x_d^2)}, \qquad \text{and} \qquad  A^!=\frac{k\langle x_1^*,\ldots, x_d^*\rangle}{(x_i^*x_j^*+x_j^*x_i^*, 1\leq i<j\leq d)},
\]
where $\dim_k(A_j)=\binom{d}{j}$ and $\dim_k(A^!_{-i})=\binom{i+d-1}{d-1}$. Thus the terms in the double complex \eqref{eq:bicomplex} are
\[
F_{i,j}=
A \otimes_k k^{\binom{i+d-1}{d-1}} \otimes_k k^{\binom{d}{j}}.
\]
Notice that for $a\leq d$ the ideal $\m^a$ of $A$ can be described as the ideal generated by all square-free monomials of degree $a$ in $A$, while for $a>d$ we have $\fm^a=0$. We compute the Betti numbers of this family of ideals using Corollary \ref{cor:betti} as follows
\begin{equation}
 \label{eq:fakeGalettobetti}
\beta_{n,n+a}(\fm^a)=\sum_{i=1}^a (-1)^{i+1}\binom{n+i+d-1}{d-1}\binom{d}{a-i}
%&=& \sum_{i=n-d}^n (-1)^{i+1}\binom{a+i+d-1}{d-1}\binom{d}{n-i}.
\end{equation}
\end{example}
Note that $(-1)^{i+1}\dbinom{n+i+d-1}{d-1}$ is equal to $(-1)^{n-1}$ times the coefficient of $t^{n+i}$ in the Taylor expansion of the rational function $\frac{1}{(1+t)^d}$ around 0. 
Similarly, $\dbinom{d}{a-i}$ is the coefficient of $t^{a-i}$ in the binomial expansion of $(1+t)^d$. 
Since $\frac{1}{(1+t)^d} \cdot (1+t)^d=1$, for $n+a>0$, the coefficient of $t^{n+a}$ in their product is 0, i.e.
\[\sum_{i=-n}^{a} (-1)^{n+i}\dbinom{n+i+d-1}{d-1}\dbinom{d}{a-i}=0.\] However, when $i<a-d$, the second binomial coefficient is 0, so this can be restated as
\[\sum_{i=a-d}^a (-1)^{n+i}\dbinom{n+i+d-1}{d-1}\dbinom{d}{a-i}=0.\]
Combined with  \eqref{eq:fakeGalettobetti}, the identity above leads to the more compact formula
\[
\beta_{n,n+a}(\fm^a)=
\begin{cases}
\sum_{i=a-d}^{0} (-1)^{i}\binom{n+i+d-1}{d-1}\binom{d}{a-i} & 1\leq a\leq d\\
0 & a \geq d+1.
\end{cases}
\]
This is consistent with $\m^a=0$ for $a>d$ and can be easier to evaluate than \eqref{eq:fakeGalettobetti} for some values of $a$. For example, setting $a=d$ yields
\[
\beta_{n,n+d}(\fm^d)=
\binom{n+d-1}{d-1}.
\]

%%%%%%%%%%%%%%%%%%%%%%%%%
\noindent {\bf Acknowledgements.}
%%%%%%%%%%%%%%%%%%%%%%%%%
Our work started at the 2019 workshop  "Women in Commutative Algebra" hosted by Banff International Research Station. We thank the organizers of this workshop for bringing our team together. We acknowledge the excellent working conditions provided by BIRS and the support of the National Science Foundation for travel through grant DMS-1934391. We thank the Association for Women in Mathematics for funding from grant NSF-HRD 1500481. 

In addition, we have the following individual acknowledgements for support: 
Faber was supported by the European Union's Horizon 2020 research and innovation programme under the Marie Sk{\l}odowska-Curie grant agreement No 789580. 
Miller was partially supported by the NSF DMS-1003384.  
R.G.'s travel was partially supported by an AMS-Simons Travel Grant.
Seceleanu was partially supported by NSF DMS-1601024. 

We thank Liana \c{S}ega for helpful comments and for bringing  \cite{Sega} to our attention and Ben Briggs for answering a question and pointing us to \cite{vandenBergh}. 

Lastly, we thank the referee for a thorough reading and especially for guiding us to other points of view, cf.\ Remarks~\ref{viakoszulduality} and \ref{viatwisting}, as well as for pointing out the papers \cite{GreenMartinezVilla} and \cite{MartinezVillaZacharia}.

%\bibliographystyle{amsalpha}
%\bibliography{biblio}

\newcommand{\etalchar}[1]{$^{#1}$}
\providecommand{\bysame}{\leavevmode\hbox to3em{\hrulefill}\thinspace}
\providecommand{\MR}{\relax\ifhmode\unskip\space\fi MR }
\providecommand{\MRhref}[2]{%
  \href{http://www.ams.org/mathscinet-getitem?mr=#1}{#2}
}
\providecommand{\href}[2]{#2}

%%%%%%%%%%%%%%%%%%%%%%%%%%%%%
%%%%%%%%%%%%%%%%%%%%%%%%%%%%%
%%%%%%%%%%%%%%%%%%%%%%%%%%%%%
\end{document}